\documentclass[reqno]{amsart}

\usepackage[fontsize=10.3 pt]{scrextend}
\usepackage[T1]{fontenc}
\usepackage{pgf,pgfarrows,pgfnodes,pgfautomata,pgfheaps,pgfshade,hyperref, amssymb}
\usepackage{amssymb}
\usepackage{tikz-cd}
\usepackage{amsmath}

\usepackage[english]{babel}
\usepackage[capitalize]{cleveref}
\usepackage{mathtools}
\usepackage[colorinlistoftodos]{todonotes}
\usepackage{soul}
\usepackage[utf8]{inputenc} 
\usepackage[T1]{fontenc}
\usepackage[normalem]{ulem}

\usepackage{tikz}
\usetikzlibrary{arrows}

\usepackage{xcolor}

\usepackage{faktor}
\usepackage{xfrac}

\usepackage{enumerate}
\usepackage{comment}

\def\R{{\mathbb R}}
\def\Z{{\mathbb Z}}

\def\N{{\mathbb N}}

\newcommand{\supp}{\mathsf{supp}}
\hypersetup{
    colorlinks,
    linkcolor={blue!30!black},
    citecolor={green!50!black},
    urlcolor={blue!80!black}
} 

\newcommand{\nnorm}[1]{\lvert\!|\!| #1|\!|\!\rvert}

\usepackage{mathrsfs} 
\usepackage{dsfont}

\usepackage{graphicx} 

\newtheorem{theorem}{Theorem}[section]
\newtheorem{proposition}[theorem]{Proposition}

\newtheorem{lemma}[theorem]{Lemma}

\newcounter{thmcounter}

\newcounter{introthmcounter}

\newtheorem{corollary}[theorem]{Corollary}
\theoremstyle{definition}

\newtheorem*{definition*}{Definition}
\newtheorem{question}[theorem]{Question}
\newtheorem*{question*}{Question}
\newtheorem*{remark*}{Remark}
\newtheorem{thmx}{Theorem}

\newcounter{proofcount}
\AtBeginEnvironment{proof}{\stepcounter{proofcount}}

\makeatletter
\def\section{\@startsection{section}{1}%
  \z@{.5\linespacing\@plus.7\linespacing}
{.8\baselineskip}%
  {\normalfont\fontsize{11}{13}\centering\bfseries}%
}

\makeatletter
\def\subsection{\@startsection{subsection}{2}%
  \z@{.4\linespacing\@plus.7\linespacing}
{.6\baselineskip}%
  {\normalfont\centering\bfseries}%
}

\makeatletter                  
\@addtoreset{claim}{proofcount}
\makeatother

\theoremstyle{remark}

\newtheorem{remark}[theorem]{Remark}

\def\N{{\mathbb N}}
\newcommand{\E}{\mathbb{E}}

\newcommand{\Folner}{F\o{}lner}

\title{A refined structure theorem for polynomial return-time sets 
in minimal systems}
\author{Ioannis Kousek}
\address{Department of mathematics, University of Warwick}
\email{ioannis.kousek@warwick.ac.uk}
\date{}

\begin{document}

\begin{abstract}
\small{We prove a refinement of a recent structural result in 
topological dynamics due to 
Glasscock, Koutsogiannis, Le, Moreira, Richter, and 
Robertson, who 
showed that in a minimal system, polynomial return-times 
differ by non-piecewise syndetic sets from those in its 
maximal infinite-step pronilfactor. In particular, we prove 
that the infinite-step pronilfactor can be replaced by the
maximal $k$-step pronilfactor of the system, where $k$ 
depends only on the given polynomials, with $k-1$ being equal to 
the Host-Kra complexity of the polynomials.}
\end{abstract}

\maketitle

\section{Introduction}

If $T:X \to X$ is a homeomorphism of a compact metric space, $X$, 
we refer to $(X,T)$ as a \textit{topological dynamical system}. 
Given nonempty sets $U_1,\ldots,U_d\subset X$ and a polynomial 
tuple $p=(p_1,\ldots,p_d)\in \Z[x]^{d}$, we are interested in 
\textit{the set of return-times of $U_1,\ldots,U_d$ along $p$} 
defined as 
\begin{equation}\label{return time set eq}
R_p(U_1,\ldots,U_d)=\{n\in \Z: T^{-p_1(n)}U_1\cap \cdots \cap T^{-p_d(n)}U_d \neq \emptyset \}.    
\end{equation}

We will focus on 
\textit{essentially distinct} polynomial tuples 
$p=(p_1,\ldots,p_d)\in \Z[x]^d$, which 
means that $p_i-p_j$ is not constant for all $1 \leq i < j\leq d$. 
In the case where all polynomials vanish at $0$, this is 
equivalent to the polynomials being pairwise distinct. 

When considering minimal systems, it was shown for 
linear polynomials 
in \cite{Glasner_Huang_Shao_Weiss_Ye}, and, in the general case in \cite{Qiu}, that the 
maximal infinite-step 
pronilfactor is \textit{characteristic} for the 
polynomial recurrence 
exhibited in \eqref{return time set eq} in the following 
sense. 
\begin{theorem}[{\text{cf.} {\cite[Theorem B]{Qiu}}}]\label{Qiu}
Let $(X,T)$ be a minimal and invertible
topological dynamical system. Denote by $(X_{\infty},T)$ its maximal infinite-step pronilfactor, and let $\pi_{\infty}: X \to X_{\infty}$ be the associated factor map. For all $d\in \N$, all nonempty, open $U_1,\ldots,U_d \subset X$, and all 
essentially distinct $p\in \Z[x]^d$, if $R_p((\pi_{\infty} U_1)^{\circ},\ldots,(\pi_{\infty} U_d)^{\circ}) \neq \emptyset$, then $R_p(U_1,\ldots,U_d) \neq \emptyset$. 
\end{theorem}
We remark that (see \cite[Lemma 3.2]{GKLMRR}) if $(Y,T)$ is any factor of a minimal system 
$(X,T)$, with factor map
$\pi:X\to Y$ (see Section 
\ref{prel dynamics} for precise definitions), then 
\begin{equation}\label{inclusion_0} 
R_p(U_1,\ldots,U_d) \subset R_p((\pi U_1)^{\circ},\ldots,(\pi U_d)^{\circ}).    
\end{equation}
Hence, Theorem \ref{Qiu}, guaranteeing that recurrence along 
$p$ can be ``lifted'' from the infinite-step pronilfactor to the
original system, is a ``surprising and highly useful converse''\footnote{quote from \cite[page 3]{GKLMRR}} to 
the above inclusion which shows that recurrence passes from any 
system to any of its factors. 

The main result in 
\cite{GKLMRR} refines this notion of characteristic factor to show 
that the difference between the return-time sets appearing in 
Theorem \ref{Qiu} cannot be large in a precise combinatorial 
sense. 

\begin{theorem}[{{\cite[Theorem A]{GKLMRR}}}]\label{GKMLRR}
Let $(X,T)$ be a minimal and invertible
topological dynamical system. Denote by $(X_{\infty},T)$ its maximal infinite-step pronilfactor, and let $\pi_{\infty}: X \to X_{\infty}$ be the associated factor map. For all $d\in \N$, all nonempty, open $U_1,\ldots,U_d \subset X$, and all 
essentially distinct $p\in \Z[x]^d$, the set 
\begin{equation}\label{set difference of return times}
R_p((\pi_{\infty} U_1)^{\circ},\ldots,(\pi_{\infty} U_d)^{\circ}) \setminus R_p(U_1,\ldots,U_d)     
\end{equation}   
is not piecewise syndetic.
\end{theorem}

Theorem \ref{Qiu} was strengthened by Huang, Shao and Ye in 
\cite{Huang_Shao_Ye} and the latter was refined by Ye and Yu in \cite{Ye_Yu} who managed to 
replace the infinite-step pronilfactor by some finite-step 
pronilfactor depending only on the polynomials. We state their 
result in a weaker form (for the original version, see Appendix \ref{appendix}). 

\begin{theorem}[{{\text{cf.} \cite[Theorem A]{Ye_Yu}}}]\label{Ye Yu}
Let $(X,T)$ be a minimal and invertible
topological dynamical system. Let $d\in \N$, and let $p\in \Z[x]^d$ be essentially distinct. Then, there is $k\in \N$ such that for all nonempty, open sets $U_1,\ldots,U_d\subset X$, if $R_p((\pi_k U_1)^{\circ},\ldots,(\pi_k U_d)^{\circ}) \neq \emptyset$, then $R_p(U_1,\ldots,U_d) \neq \emptyset$,
where $\pi_k: X\to X_k$ denotes the factor onto the maximal $k$-step pronilfactor of $X$. 
\end{theorem}

In view of Theorem \ref{Ye Yu}, it was speculated in 
\cite{GKLMRR} 
whether some finite-step pronilfactor, with the step 
depending 
only on the polynomial family, is characteristic in the sense 
of Theorem \ref{GKMLRR}. 

\begin{question}[{{\cite[Question 5.1]{GKLMRR}}}]\label{finite step nilfactor in GKLMRR}
Given $d\in \N$ and an essentially distinct tuple $p\in \Z[x]^d$, is there $k\in \N$ such that for all minimal and invertible
topological dynamical systems $(X,T)$ and all nonempty, open sets $U_1,\ldots,U_d \subset X$, 
\begin{equation}\label{finite set nilfactor eq}
R_p((\pi_k U_1)^{\circ},\ldots,(\pi_k U_d)^{\circ}) \setminus R_p(U_1,\ldots,U_d)   
\end{equation}
is not piecewise syndetic, where $\pi_k:X\to X_k$ denotes the factor map to the maximal $k$-step pronilfactor of $(X,T)$?   
\end{question}

Our main result answers this in the positive, by juxtaposing 
ideas from \cite{GKLMRR}, \cite{Ye_Yu}, and the structure 
theory of multiple ergodic averages (\cite{Host_Kra_nonconventional_averages_nilmanifolds:2005}, \cite{Ziegler07}), via the use of 
Host-Kra uniformity seminorms, in order to control polynomial 
averages that naturally 
appear in our ergodic theoretic approach to the problem.

\begin{thmx}\label{Main theorem}
Let $d\in \N$ and an essentially distinct tuple $p\in \Z[x]^d$. 
Then, there is $k\in \N$ such that for all minimal and invertible
topological dynamical systems $(X,T)$ and all nonempty, open 
sets $U_1,\ldots,U_d \subset X$,
\begin{equation}\label{finite set nilfactor eq}
R_p((\pi_k U_1)^{\circ},\ldots,(\pi_k U_d)^{\circ}) \setminus R_p(U_1,\ldots,U_d)   
\end{equation}
is not piecewise syndetic, where $\pi_k:X\to X_k$ denotes the 
factor map to the maximal $k$-step pronilfactor of $(X,T)$. 
Moreover, $k$ is such that $k-1$ is the Host-Kra 
complexity of the polynomial tuple $p\in \Z[x]^d$.  
\end{thmx}

The last conclusion of the statement of Theorem \ref{Main theorem}, specifying the optimal 
integer $k$, is explained in the end of Section \ref{final proof}.

In order to prove Theorem \ref{Main theorem}, we also 
establish the 
following result which may be of independent interest. 

\begin{thmx}\label{k factor is characteristic intro}
Let $d,s\in \N$. Then, there exists $k=k(d,s)\in \N$ with the following 
property. For any minimal infinite-step pronilsystem, 
$(X_{\infty},T)$, any $m\in \N$ and tuples 
$p_j=(p_{j,1},\ldots,p_{j,d})\in \Z[x]^{d}$, $j=1,\ldots,m$ 
of 
essentially distinct 
polynomials vanishing at $0$ of degree at most $s$, if $U_{j,i}\subset X_{\infty}$, for $i=0,1,\ldots,d$ and 
$j=1,\ldots,m$ are open sets such that 
$$\pi_k(U_{j,0}) \cap \cdots \cap \pi_k(U_{j,d}) \neq \emptyset, $$
then there is $n\in \N$ such that 
\begin{equation}\label{same return time}
U_{j,0} \cap \bigcap_{i=1}^d T^{-p_{j,i}(n)}U_{j,i} \neq \emptyset,    
\end{equation}
for each $j=1,\ldots,m$, where $\pi_k:X_{\infty}\to X_k$ denotes the factor map to the maximal $k$-step pronilfactor.
\end{thmx}

Finally, we would like to remark that the aforementioned 
results of \cite{Huang_Shao_Ye} and Theorem \ref{Qiu} were 
recently extended to the product of finitely many minimal 
systems in \cite{QiuXuYeYu} and \cite{WuYu}, respectively (in 
the sense of simultaneous recurrence).
It turns out that Theorems \ref{GKMLRR} and \ref{Main theorem} can also be extended in the same direction; we merely present an overview of the argument as the details are beyond the scope of this article. It can be shown that \cite[Theorem 1.1]{WuYu} -- using the notation therein -- is equivalent to the statement that if $A:=\bigcap_{i=1}^k R_p ((\pi_i V_1^{(i)})^{\circ} , \ldots , (\pi_i V_d^{(i)})^{\circ} ) \neq \emptyset $, then also 
$B:= \bigcap_{i=1}^k R_p (V_1^{(i)} , \ldots , V_d^{(i)}) \neq \emptyset $, 
where $\pi_i: X_i \to Y_i$ are projections to any factor $Y_i$ that contains the infinite-step pronilfactor of $X_i$ and $p=(p_1,\ldots,p_d)\in \Z[x]^d$ is an essentially distinct tuple of polynomials vanishing at $0$. In this context, the analogue of Theorem \ref{GKMLRR} would be that $A\setminus B$ is not piecewise syndetic. Note, however, that  
$(A\setminus B) \subset \bigcup_{i=1}^k (A_i\setminus B_i),$
where $A_i:=R_p ((\pi_i V_1^{(i)})^{\circ} , \ldots , (\pi_i V_d^{(i)})^{\circ} )$ and 
$B_i=R_p (V_1^{(i)} , \ldots , V_d^{(i)})$. 
Then, each of the sets $A_i\setminus B_i$ is not piecewise syndetic by Theorem \ref{GKMLRR}, and thus the same holds for $A \setminus B$ by Brown's lemma (see Lemma \ref{Brown} below). 
The analogue of Theorem \ref{Main theorem} in this direction (with appropriate finite-step pronilfactors being characteristic), can be deduced using \cite[Theorem 3.4]{WuYu} and the main results of the present paper. It would be interesting to see whether these results -- starting from those of \cite{QiuXuYeYu, Ye_Yu} -- extend to (countably) infinite products of minimal transformations as well. 

\subsection*{Acknowledgments}
The author thanks Joel Moreira and Andreas 
Koutsogiannis for comments on an earlier draft of the paper. 
The author is also grateful to Professor Xiangdong Ye for 
reading the arguments of Appendix \ref{App B} and providing some 
helpful references. 

\section{Preliminaries}

In this section, we collect some definitions and results from 
combinatorics, topological dynamics and ergodic theory, that will 
be used in the proof of the main result.  

\subsection{Combinatorics}

We denote by $\Z$ and $\N$ the set of integers and positive 
integers, respectively. For $A\subset \Z$ and $m\in \Z$ we let 
$A-m=\{n\in \Z: n+m\in A\}$. 

The following notions of largeness 
for subsets of $\Z$ are very important and appear frequently in 
the intersection of Ramsey theory and topological dynamics. A set 
$A\subset \Z$ is 
\begin{itemize}
    \item \textit{syndetic} if there exists $k\in \N$ such that 
    $\bigcup_{j=1}^k A-j=\Z$
    \item \textit{thick} if for all $N\in \N$, there exists $m\in \Z$ such that $\{m,m+1,\ldots,m+N\} \subset \Z$
    \item \textit{piecewise syndetic} if there exists $k\in \N$ such that $\bigcup_{j=1}^{k} A-j$ is thick.
\end{itemize}

It is well-known and easy to deduce from the definitions that a 
set is thick if and only if it intersects every syndetic set, or 
equivalently, if its complement is not syndetic. The roles of 
thick and syndetic sets in the previous sentence can be 
interchanged. Moreover, a set is piecewise syndetic if and only if 
it is the intersection of a syndetic and a thick set. Note also 
that shifts of syndetic, thick or piecewise syndetic sets are 
syndetic, thick or piecewise syndetic, respectively. 

We will use the fact that piecewise syndetic sets possess the 
Ramsey property, often referred to as Brown's lemma (see, e.g.  \cite[Theorem 4.4]{Hindman_Strauss98})

\begin{lemma}\label{Brown}
Let $A\subset \Z$ be piecewise syndetic. If 
$A=A_1 \cup \cdots \cup A_r$, then one of the sets $A_i$ is also 
piecewise syndetic. 
\end{lemma}

Next is a simplified version of \cite[Lemma 2.1]{GKLMRR} that will be
utilised in the paper. 

\begin{lemma}\label{A-B pws}
Let $B\subset A\subset \Z$. If the set $\bigcap_{n\in F}B-f$ is syndetic for any finite $F\subset A$, then $A\setminus B$ is not piecewise syndetic. 
\end{lemma}

\begin{proof}
Towards a contradiction, suppose there exists $\ell \in \N$ such 
that $\bigcup_{i=1}^{\ell}(A\setminus B)-i$ is thick. Let 
$F\subset A$ be a subset of $A$ contained in an interval $I$ 
of length $\ell$, which has maximal cardinality among all 
subsets of $A$ with this property. Namely, there is no subset 
$C\subset A$ contained in an interval of length $\ell$ for 
which $|C|>|F|$. By assumption, $S=\bigcap_{n\in F}B-f$ is 
syndetic, and $\bigcup_{i\in I}(A\setminus B)-i$ is thick as 
the shift of a thick set. 
The rest of the argument proceeds exactly as in 
\cite{GKLMRR}.    
\end{proof}

A \Folner\ sequence $\Phi=(\Phi_N)_{N\in \N}$ in $\Z$ (or $\N$, analogously)  
is a sequence of nonempty finite subsets of $\Z$ which are 
asymptotically invariant under any shift, that is, 
$$\lim_{N\to \infty} \frac{|\Phi_N \cap (\Phi_N+1)|}{|\Phi_N|}=1.$$
We say that a set $A\subset \Z$ has positive upper Banach density if and only if there is a \Folner\ sequence $\Phi=(\Phi_N)_{N\in \N}$ such that 
$$ \lim_{N\to \infty} \frac{|A\cap \Phi_N|}{|\Phi_N|}>0.$$

We conclude this subsection with a simple telescoping trick which we will use 
to replace Ces\`aro averages of products of sequences. 
Namely, for general bounded sequences $a_1,\ldots,a_d,b_1,\ldots,b_d:\Z\to \R$, it 
holds that 
\begin{equation}\label{eq:18}
a_1(n)\cdots a_d(n)- b_1(n)\cdots b_d(n)=\sum_{j=1}^d a_1(n)\cdots a_{j-1}(n)(a_j(n)-b_j(n))b_{j+1}(n)\cdots b_d(n),
\end{equation} 
where $a_0, b_{d+1}$ are both the constant sequence $1$. 

\subsection{Topological dynamics}\label{prel dynamics}

By a \textit{topological dynamical system}, or \textit{system} for short, we mean a pair $(X,T)$, 
where $X$ is a compact metric space and $T: X\to X$ a 
homeomorphism. 
Throughout, we only consider invertible transformations, which is 
essential because the regionally proximal relations used to 
develop the theory of maximal pronilfactors do not seem to have 
been defined in the literature for non-invertible transformations (or $\N$-actions). 

A system $(X,T)$ is minimal if every point $x$ has a dense orbit, 
i.e., $\{T^nx: n\in \Z\}$ is dense in $X$. 

A \textit{factor map} between two systems $(X,T)$ and $(Y,S)$ is a 
continuous surjection $\pi:X\to Y$ satisfying 
$\pi \circ T=S\circ \pi$, meaning that $\pi(T(x))=S(\pi(x))$, 
for all $x\in X$. In this case, we call $(Y,S)$ a \textit{factor} 
of $(X,T)$ or the latter an \textit{extension} of the former and 
we sometimes use the same letter for the transformations $T$ and 
$S$, writing $(Y,T)$ instead of $(Y,S)$. If $(Y_1,T)$ and $(Y_2,T)$ are two factors of $(X,T)$ with corresponding factor maps $\pi_1:X\to Y_1$ and $\pi_2:X\to Y_2$ and there is another factor map $\pi:Y_1\to Y_2$ such that $\pi_2=\pi \circ \pi_1$, we say that $(Y_1,T)$ \textit{contains} $(Y_2,T)$. 

When $\phi:X\to Y$ is a map between compact metric spaces, and $k\in \N$, we write $\phi^{(k)}$ for the map $\phi \times \cdots \times \phi: X^k \to Y^k$, the diagonal of $X^k$ is $\Delta_k(X)=\{x^{(k)}=(x,\ldots,x)\in X^k: x\in X\}$ and when $(X,T)$ is a system, we write $(X^k,T^k)$ for the $k$-fold product system $(X\times \cdots \times X, T\times \cdots \times T)$. 

If $(X,T)$ is a system, a pair $(x,y)\in X^2$ is called \textit{proximal} if $\inf_{n\in \Z}d(T^nx,T^ny)=0$, where $d$ is 
a compatible metric for $X$. The pair is otherwise \textit{distal} and the system $(X,T)$ is called \textit{a distal system} if all pairs of distinct points are distal. We denote by $\textbf{P}(X,T)\subset X^2$ the set of all proximal pairs of $(X,T)$. 

A factor map $\pi:X\to Y$ is almost $1-1$ if there exists a dense $G_{\delta}$ subset $X' \subset X$ such that $\pi^{-1}(\{\pi(x)\})=\{x\}$ for all $x\in X'$. In this case, we say that $(X,T)$ is \textit{an almost $1-1$ extension} of $(Y,T)$. Moreover, the factor map $\pi$ is called \textit{distal} if $\pi(x)=\pi(y)$ for $x\neq y$ implies that $(x,y) \notin \textbf{P}(X,T)$.

It is well-known (see, for example, \cite[Lemma 2.9]{Glasscock_Koutsogiannis_Richter}) 
that a factor map between minimal systems $\pi: X\to Y$ is \textit{semiopen}, meaning that $(\pi(U))^{\circ}$ is nonempty for every nonempty open set $U\subset X$. The factor map is \textit{open} if it maps open sets to open sets. 

Let $(X_j,T_j)_{j\in \N}$ be a sequence of topological dynamical systems endowed with factor maps $\pi_{j,j+1}: X_{j+1}\to X_j$ for each $j\in \N$. An \emph{inverse limit} of this sequence of systems is the unique system $(X,T)$ (up to topological isomorphism) endowed with factor maps $\pi_j:X\to X_j$, $j\in \N$, such that
\begin{enumerate}[(i)]
    \item \label{point inverse 1}  $\pi_{j+1}=\pi_{j,j+1} \circ \pi_j$, for every $j\in \N$ and 
      \item $\bigcup_{i\in \N} \{f \circ \pi_i : f\in C(X_i)\}$ is dense in $C(X)$.
\end{enumerate}
We also remark that the inverse limit satisfies the following \textit{universal property}: If $(Y, T)$ is a system and for $i\in \N$, $p_i:Y \to X_i$ is a factor map such that such that $p_i=\pi_{i,j} \circ p_j$, for $i\leq j$, then there exists a unique factor map $p:Y\to X$, such
that $p_i = \pi_i \circ p$, for every $i\in \N$.

\subsection{Ergodic theory}

A \textit{measure preserving system} $(X,\mu,T)$ is a triple where $(X,T)$ is a topological dynamical system and $\mu$ is a $T$-invariant Borel probability measure, namely $T^{*}\mu=\mu$. The system $(X,\mu,T)$ is \textit{ergodic} if $T^{-1}A=A$ implies that $\mu(A)\in \{0,1\}$ for any Borel set $A$. The system $(X,T)$ is \textit{uniquely ergodic} if there exists only one $T$-invariant measure $\mu$, which makes the system $(X,\mu,T)$ automatically ergodic. 

If $(X,\mu,T)$ is measure preserving, the \textit{support of $\mu$} is 
defined as the smallest full measure subset of $X$, or 
equivalently, as 
$$\supp(\mu)=\{x\in X: \text{for each open neighborhood U of}\ x,\ \mu(U)>0\}.$$

Let $(X, \mu,T)$ and $(Y,\nu,S)$ be measure preserving systems. We say that $(Y, \nu,S)$ (or simply $Y$) is a \emph{factor} of $(X, \mu,T)$ (or simply $X$) if there exists a measurable map $\pi \colon X \to Y$, called a \emph{factor map}, such that $\pi^{*}\mu = \nu$ and $\pi \circ T = S \circ \pi$, up to null-sets. Again, we use the same letter for the transformation of a factor and its extension, that is, we write $(Y,\nu,T)$. Given a factor map $\pi: X\to Y$ and a function $f\in L^2(\mu)$, 
we consider the conditional expectation 
$\mathbb{E}(f | \pi^{-1}\mathcal{B}(Y))$, where $\mathcal{B}(Y)$ is the Borel $\sigma$-algebra on $Y$. We write $\mathbb{E}(f | \mathcal{B}(Y))$ or $\mathbb{E}(f | Y)$ for the $\mathcal{B}(Y)$-measurable function that satisfies $\mathbb{E}(f | \mathcal{B}(Y))\circ \pi = \mathbb{E}(f | \pi^{-1}\mathcal{B}(Y))$.

Let $(X_j,\mu_j,T_j)_{j\in \N}$ be a sequence of measure preserving systems endowed with factor maps $\pi_{j,j+1}: X_{j+1}\to X_j$ for each $j\in \N$. An \emph{inverse limit} of this sequence of systems is the unique system $(X,\mu,T)$ (up to isomorphism) endowed with the factor maps $\pi_j:X\to X_j$, $j\in \N$ such that
\begin{enumerate}[(i)]
    \item \label{point inverse 1}  $\pi_{j+1}=\pi_{j,j+1} \circ \pi_j$, for every $j\in \N$ and 
    \item For $1\leq p <\infty$, $\bigcup_{i\in \N} \{f \circ \pi_i : f\in L^p(\mu_i)\}$ is dense in $L^p(\mu)$.
\end{enumerate}

Ergodic systems are fundamental as building blocks of all systems 
in the sense of the ergodic decomposition theorem (see, for 
example, \cite[Theorem 6.2]{Einsiedler_Ward11}). 

\begin{theorem}[\textbf{Ergodic decomposition}]\label{ergodic decomposition}
Let $(X,\mu,T)$ be a measure preserving system and let $M_T(X)$ denote the set of $T$-invariant Borel probability measures on $X$. Then there exist a Borel probability space $(Y,\mathcal{B}(Y),\nu)$ and measures $\mu_y \in M_T(X)$ such that,
\begin{enumerate}[(i)]
\item $y \mapsto \mu_y$ is measurable, that is, $y \mapsto \int_X f\ d\mu_y$ is measurable for every $f\in L^1(\mu)$  
\item $\mu_y$ is ergodic for $\nu$-almost every $y\in Y$, and 
\item $\mu=\int_Y \mu_y\ d\nu(y)$, which means that for every $f\in L^1(\mu)$ we have
$$\int_X f(x)\ d\mu(x)=\int_Y \left( \int_X f(x) \ d\mu_y(x) \right)\ d\nu(y).$$  
\end{enumerate} 
\end{theorem}

\subsection{Nilsystems and pronilsystems}

A $k$-step nilsystem $(X,T) = (G/\Gamma,T)$ is a dynamical system given by a $k$-step nilpotent Lie group $G$, with $\Gamma$ being a co-compact discrete subgroup, $T \colon X \to X$ denoting the left translation by a fixed element of $G$. If $\mu$ denotes the Haar measure of $X$, then $(X,\mu,T) = (G/\Gamma,\mu,T)$ is a measure preserving system. For the $k$-step nilsystem $(X,\mu,T)$, the properties of ergodicity, minimality and unique ergodicity are equivalent. Moreover, if $x \in X$ is any point, the orbit closure $\overline {\{ T^n x \colon n \in \Z \}} $ is a minimal and uniquely ergodic $k$-step subnilsystem. Also, nilsystems are always distal. An inverse limit (in either the topological or the measure theoretic sense) of $k$-step nilsystems is called a $k$-step pronilsystem. An \emph{infinite-step pronilsystem} is an inverse limit of a sequence of finite-step pronilsystems of increasing and unbounded step.

All the above properties are preserved under inverse limits of $k$-step nilsystems, so they also hold for $k$-step pronilsystems (i.e. distality and the equivalence between minimality, ergodicity and unique ergodicity). For these classical facts and other details we refer the reader to \cite[Chapter 11 and Chapter 13]{Host_Kra_nilpotent_structures_ergodic_theory:2018}.  

Let $(X,\mu,T)$ be an ergodic measure preserving system. Then, for 
every $k \in \N$, $(X, \mu,T)$ has a maximal factor isomorphic to 
a $k$-step pronilsystem, called the $k$\emph{-step pronilfactor}. 
We shall denote this factor by $(Z_k,m_k,T)$. The inverse limit of 
$(Z_k,m_k,T)$ is the \emph{infinite-step pronilfactor of} $X$ and 
we denote it by $(Z_{\infty}, m,T)$. The infinite-step 
pronilfactor of $X$ coincides with the maximal factor of $X$ 
isomorphic to an infinite-step pronilsystem.  

In \cite{Host_Kra_Maass_nilsequences} and 
\cite{Shao_Ye_regionally_prox_orderd:2012}, the theory of 
regionally proximal relations for pairs of points in a topological system 
$(X,T)$ was developed, ultimately leading to a topological 
analogue of the pronilfactors $(Z_k,m_k,T)$. Specifically, it was 
shown that for every $k\in \N$, every minimal system $(X,T)$ 
admits a (unique) factor that is a $k$-step pronilsystem and which 
contains all $k$-step nilfactors of $(X,T)$ as factors. We denote 
this factor by $(X_k,T)$ and it is called the maximal $k$-step 
pronilfactor of $(X,T)$. The (topological) inverse limit of the sequence 
$((X_k,T))_{k\in \N}$, denoted by $(X_{\infty},T)$, is the maximal 
infinite-step pronilfactor of $(X,T)$.  

\subsection{Uniformity seminorms}

In their seminal paper 
\cite{Host_Kra_nonconventional_averages_nilmanifolds:2005}, Host 
and Kra  showed that it is possible to describe the measurable 
pronilfactors presented in the previous section by studying 
uniformity seminorms.  

Let $(X, \mu,T)$ be a measure preserving system and $f \in L^{\infty}(\mu)$. The $k$-uniformity seminorms of $f$, $\nnorm{f}_{U^{k}(X,\mu,T)}$, are defined inductively as follows: 
    \begin{equation*}
        \nnorm{f}_{U^0(X, \mu,T)} =  \int f d\mu 
    \end{equation*}
    and 
    \begin{equation} \label{eq def seminorm}
        \nnorm{ f}_{U^{k+1}(X, \mu,T)}^{2^{k+1}}= \lim_{H \to \infty }\frac{1}{H} \sum_{h=1}^{H} \nnorm{f \cdot \overline{T^hf}}_{U^k(X, \mu,T)}^{2^k}.
    \end{equation}


The main theorem in \cite{Host_Kra_nonconventional_averages_nilmanifolds:2005} shows that the $k$-step pronilfactor, $Z_{k}$, of an ergodic system $(X, \mu,T)$ is characterized by the property 
    \begin{equation*} 
        \nnorm{f}_{U^{k+1}(X, \mu,T)} = 0 \iff \E(f \mid Z_{k}) =0,
    \end{equation*}
    for any $f\in L^{\infty}(\mu)$. In particular, this implies that 
    \begin{equation}\label{uniformity of orthocomplement}
    \nnorm{f-\mathbb{E}(f | Z_k)\circ \pi_k}_{U^{k+1}(X, \mu,T)} = 0,
    \end{equation}
    where $\pi_k$ denotes the factor map $\pi_k: X \to Z_k$.

The following is well-known and easy to prove using induction and the definition of the seminorms given in \eqref{eq def seminorm}.

\begin{lemma}\label{uniformity in tensor product}
Let $(X,\mu,T)$ be a measure preserving system. For $s\in \N$ and $f\in L^{\infty}(\mu)$, we have that 
$$ \nnorm{f}_{U^{s+1}(X,\mu,T)}=0 \implies \nnorm{f \otimes f}_{U^s(X\times X, \mu\times \mu, T\times T)}=0.$$
\end{lemma}

It is shown in the proof of \cite[Prop. 2.2]{Host_Kra09} that if 
$(X,\mu,T)$ is a measure preserving system, $k\in \N$ and 
$f\in L^{\infty}(\mu)$, then 
$$\nnorm{f}_{U^k(X,\mu,T)}^{2^k}=\int_Y \nnorm{f}_{U^k(X,\mu_y,T)}^{2^k}\ d\nu(y),$$ 
where $y\mapsto \mu_y$ is the ergodic decomposition of $\mu$ as in
Theorem \ref{ergodic decomposition}. Using this fact the next 
result follows immediately. 

\begin{lemma}\label{uniformity for ergodic components}
Let $(X,\mu,T)$ be a measure preserving system, $k\in \N$ and 
$f\in L^{\infty}(\mu)$. If 
$\nnorm{f}_{U^k(X,\mu,T)}=0$, then for $\nu$-a.e. ergodic component $\mu_y$ of $\mu$ it holds that $\nnorm{f}_{U^k(X,\mu_y,T)}=0$.      
\end{lemma}

Finally, we will use a special case of Leibman's theorem 
\cite[Theorem 3]{Leibman05} establishing 
control of general polynomial multiple ergodic averages by 
uniformity seminorms of finite order. 

\begin{theorem}\label{Leibman}
Let $b,d\in \N$. Then, there exists 
$k=k(b,d)\in \N$ such that, for any essentially distinct 
tuple $p=(p_1,\ldots,p_d)\in \Z[x]^d$ 
of polynomials of degree at most $b$, 
for any ergodic system $(X,\mu,T)$, and all 
$f_1,\ldots,f_d\in L^{\infty}(\mu)$ with 
$\nnorm{f_j}_{U^k(X,\mu,T)}=0$, for some $j\in \{1,\ldots,d\}$, it holds that 
$$\lim_{N\to \infty} \frac{1}{|\Phi_N|}\sum_{n\in \Phi_N} T^{p_1(n)}f_1 \cdots T^{p_d(n)}f_d=0,$$
in $L^2(\mu)$, for any \Folner\ sequence $\Phi=(\Phi_N)$ in $\Z$.
\end{theorem}

\section{Proof of main theorem}\label{sec main proof}

In this section we present the proof of Theorem 
\ref{Main theorem}. First, we reduce to the case of infinite-step 
pronilsystems. Then, in Section \ref{key result} we prove 
Theorem \ref{k factor is characteristic intro} (stated below 
as Theorem \ref{k factor is characteristic}), 
expanding on some techniques introduced in \cite{Glasner_Huang_Shao_Ye_regionally_arithmetic_prog_nil:2020} and then used in \cite{Ye_Yu}; in 
particular, our proof uses topological properties of distal 
extensions and the theory of characteristic factors for polynomial 
multiple ergodic averages. Finally, we apply ideas from \cite{GKLMRR} 
with Lemma \ref{A-B pws} in the epicentre, in order to complete the proof 
of the main result. 

\subsection{A reduction to infinite-step pronilsystems}\label{reduction sec}

In view of Theorem \ref{GKMLRR}, Lemma \ref{Brown} and a 
technical argument which we explain in Appendix \ref{App B}, 
the 
proof of Theorem \ref{Main theorem} is implied by 
the following result, the proof of which in turn becomes the 
goal of this section. 

\begin{theorem}\label{finite step characteristic for infinite step systems}
Let $d\in \N$ and an essentially distinct tuple 
$p\in \Z^d[x]$. Then, there exists $k\in \N$ such that for all minimal infinite-step pronilsystems, $(X_{\infty},T)$, and all nonempty, open sets $U_1,\ldots,U_d \subset X_{\infty}$, 
\begin{equation}\label{finite set nilfactor eq 2}
R_p(\pi_k U_1,\ldots,\pi_k U_d) \setminus R_p(U_1,\ldots,U_d) 
\end{equation}
is not piecewise syndetic, where $\pi_k:X_{\infty} \to X_k$ denotes the factor map to the maximal $k$-step pronilfactor of $(X_{\infty},T)$.   
\end{theorem}

\begin{remark*}
We remark that the need to consider set interiors in 
\eqref{finite set nilfactor eq 2} disappears in this case, as opposed to \eqref{finite set nilfactor eq}, 
because the factor map $\pi_k:X_{\infty}\to X_k$ is open as a 
factor map between minimal and distal systems (see, e.g. \cite[Chapter 7]{Auslander_minimal_flows_and_extensions:1988}). 
\end{remark*}

To see the relation between Theorem \ref{finite step characteristic for infinite step systems} 
and Theorem \ref{Main theorem}, recall the 
notation in the statement of the latter and 
further let $\phi_j:X_{\infty}\to X_j$ denote 
the factor map from $X_{\infty}$ 
to $X_j$, for each $j\in \N$.   

By the universal property of 
inverse limits and the definition 
of $X_{\infty}$ as the inverse 
limit of the maximal pronilfactors 
$((X_k,T))_{k\in \N}$ of $(X,T)$, 
we have that $\pi_k=\phi_k \circ \pi_{\infty}$, 
where $\phi_k: X_{\infty} \to X_k$ denotes the 
factor map. Then, letting $V_i=(\pi_{\infty}U_i)^{\circ}$, for $i=1,\ldots,d$, we see that
\begin{align}\label{inclusion} R_p(\phi_k V_1,\ldots,\phi_k V_d) & \setminus R_p(U_1,\ldots,U_d)  \subset \nonumber \\
\left( R_p(\phi_k V_1,\ldots,\phi_k V_d) \setminus R_p(V_1,\ldots,V_d)\right) & \cup  \left(  R_p(V_1,\ldots,V_d) \setminus R_p(U_1,\ldots,U_d) \right).    
\end{align}
But, if $k\in \N$ is the integer from Theorem \ref{finite step characteristic for infinite step systems}, the last union in \eqref{inclusion} is a 
non-piecewise syndetic set by Theorems \ref{finite step characteristic for infinite step systems} and \ref{Main theorem} and Lemma \ref{Brown}. 
However, observe that, as $\phi_k$ is an open map, it holds that 
$$(\pi_kU_i)^{\circ}=(\phi_k\circ \pi_{\infty}U_i)^{\circ} \supset \phi_k(\pi_{\infty}U_i)^{\circ},$$
and thus
$$R_p((\pi_k U_1)^{\circ},\ldots,(\pi_k U_d)^{\circ}) \supset R_p(\phi_k V_1,\ldots,\phi_k V_d).$$
In particular, Theorem \ref{Main theorem} is 
slightly stronger than the direct corollary of Theorem \ref{finite step characteristic for infinite step systems} and Theorem \ref{GKMLRR} which we just explained. We leave the details of the actual verification for Appendix \ref{App B}.

As was correctly speculated in \cite[Section 5.1]{GKLMRR}, the proof of Theorem 
\ref{finite step characteristic for infinite step systems} (and thus of Theorem \ref{Main theorem}) is
related to the results of Ye and Yu in \cite{Ye_Yu}. Although we do not 
use their main result -- see Theorem \ref{Ye Yu} -- as a black box, 
our proof is inspired by their arguments (which in turn seem 
to be inspired by Glasner, Huang, Shao and Ye's earlier work 
\cite{Glasner_Huang_Shao_Ye_regionally_arithmetic_prog_nil:2020}). 

In fact, the factor $k$ we find in Theorem \ref{finite step characteristic for infinite step systems} is one step larger than the factor arising in Theorem \ref{Ye Yu}. This is optimal in the sense that our characteristic factor has to be -- in general -- one step larger than the optimal characteristic factor in the sense of Theorem \ref{Ye Yu}. Indeed, as shown in Proposition $5.1$ of \cite{GKLMRR}, the maximal equicontinuous factor, $X_1$, is not characteristic in the sense of Theorem \ref{finite step characteristic for infinite step systems} for the polynomials $n\mapsto 0$, $n\mapsto n$, $n\mapsto 2n$, but it is characteristic in the sense of Theorem \ref{Ye Yu}. Nevertheless, our proof shows that the maximal $2$-step pronilfactor, $X_2$, is characteristic in the sense of Theorem \ref{finite step characteristic for infinite step systems} for these polynomials.

\subsection{A key intermediate result}\label{key result} 

In order to prove Theorem 
\ref{finite step characteristic for infinite step systems} we 
aim to apply the method that was developed in \cite{GKLMRR} 
in order to show that the 
difference of the return times sets in \eqref{set difference of return times} is not piecewise syndetic. To 
facilitate this, the following key result seems necessary. 

\begin{theorem}\label{k factor is characteristic}
Let $d,s\in \N$. Then, there exists $k=k(d,s)\in \N$ with 
the following property. For any minimal infinite-step 
pronilsystem $(X_{\infty},T)$, any $m\in \N$ and $p_j=(p_{j,1},\ldots,p_{j,d})\in \Z[x]^{d}$, $j=1,\ldots,m$, 
of essentially distinct 
polynomials vanishing at $0$ of degree at most $s$, if $U_{j,i}\subset X_{\infty}$, for $i=0,1,\ldots,d$ and 
$j=1,\ldots,m$, are open sets such that 
$$\pi_k(U_{j,0}) \cap \cdots \cap \pi_k(U_{j,d}) \neq \emptyset, $$
then there is $n\in \N$ such that 
\begin{equation}\label{same return time}
U_{j,0} \cap \bigcap_{i=1}^d T^{-p_{j,i}(n)}U_{j,i} \neq \emptyset,    
\end{equation}
for each $j=1,\ldots,m$.
In fact, the set of $n\in \Z$ for which the intersection in \eqref{same return time} is 
nonempty has positive upper Banach density. Here $\pi_k:X_{\infty}\to X_k$ denotes the factor map to the maximal $k$-step pronilfactor.
\end{theorem}

We stress that the utility of Theorem 
\ref{k factor is characteristic} lies in the fact that it 
guarantees the same return-time along different polynomial tuples 
and different families of open sets. This is ultimately the reason 
why the step of the characteristic factor increases by $1$ 
compared to Theorem \ref{Ye Yu}, which essentially 
corresponds to the case $m=1$ in the statement of 
Theorem \ref{k factor is characteristic}.

To prove Theorem \ref{k factor is characteristic} we will use 
the following results; a corollary of Leibman's Theorem \ref{Leibman} and a corollary of Bergelson-Leibman's polynomial multiple recurrence theorem.

\begin{theorem}\label{k+1 factor is characteristic for ergodic averages}
Let $b,d \in \N$ and $p_1,\ldots,p_d\in \Z[x]$ be essentially 
distinct polynomials of degree at most $b$. 
Then, there is $s\in \N$ such that for any ergodic system 
$(X,\mu,T)$, and any $f_0,f_1,\ldots,f_d\in L^{\infty}(\mu)$, if 
$\nnorm{f_j}_{U^{s}(X,\mu,T)}=0$, for some $j\in \{1,\ldots,d\}$, we have that
\begin{equation}\label{ergodic averages for k+1}
\lim_{N\to \infty} \frac{1}{|\Phi_N|}\sum_{n\in \Phi_N} \left| \int_{X} f_0 \cdot T^{p_1(n)}f_1 \cdots T^{p_d(n)}f_d\ d\mu \right| =0,     
\end{equation}
where $\Phi=(\Phi_N)$ is any \Folner\ sequence in $\Z$.
\end{theorem}

\begin{proof}
Let $k=k(b,d)\in \N$ be the integer arising from Theorem \ref{Leibman} and $s=k+1$. 
Since the sequence 
$$\left( \left| \int_{X} f_0 \cdot T^{p_1(n)}f_1 \cdots T^{p_d(n)}f_d\ d\mu \right| \right)_{n\in \Z}$$
is bounded, \eqref{ergodic averages for k+1} follows from 
$$\lim_{N\to \infty} \frac{1}{|\Phi_N|}\sum_{n\in \Phi_N} \left| \int_{X} f_0 \cdot T^{p_1(n)}f_1 \cdots T^{p_d(n)}f_d\ d\mu \right|^2 =0. $$
The latter is equivalent to 
\begin{equation}\label{ergodic averages for k+1 squared}
\lim_{N\to \infty} \frac{1}{|\Phi_N|}\sum_{n\in \Phi_N} \int_{X\times X} (f_0\otimes f_0) \cdot (T\times T)^{p_1(n)}(f_1\otimes f_1) \cdots (T\times T)^{p_d(n)}(f_d\otimes f_d)\ d(\mu\times \mu) =0.     
\end{equation}
Now, assuming that $\nnorm{f_j}_{U^{s}(X,\mu,T)}=0$, for some $j\in \{1,\ldots,d\}$, Lemma \ref{uniformity in tensor product}, the ergodic decomposition of $\mu \times \mu$ and Lemma \ref{uniformity for ergodic components} together imply that 
\begin{equation}\label{uniformity in tensor product eq}
\nnorm{f_j\otimes f_j}_{U^k(X\times X,(\mu\times \mu)_y, T\times T)}=0,    
\end{equation}
for $\nu$-a.e. $y\in Y$ in the notation of Theorem \ref{ergodic decomposition}. The proof is complete because \eqref{ergodic averages for k+1 squared} follows from
$$\int_Y \lim_{N\to \infty} \frac{1}{|\Phi_N|}\sum_{n\in \Phi_N} \int_{X\times X} (f_0\otimes f_0) \cdot (T\times T)^{p_1(n)}(f_1\otimes f_1) \cdots (T\times T)^{p_d(n)}(f_d\otimes f_d)\ d(\mu\times \mu)_y\ d\nu(y) = 0,$$
which holds by an application of Theorem \ref{Leibman} to each ergodic component for which \eqref{uniformity in tensor product eq} is true. 
\end{proof}

\begin{theorem}{{\cite[Theorem A]{Bergelson_Leibman96}}}\label{Bergelson-Leibman}
Let $(X,\mathcal{B},\mu)$ be a probability space and $T_1,\ldots,T_t$ be commuting, invertible, measure preserving transformations. Let $p_{1,1},\ldots,p_{1,t},p_{2,1},\ldots,p_{2,t},\ldots,p_{k,1},\ldots,p_{k,t}$ be polynomials with rational coefficients taking on integer values and vanishing at $0$, and let $A\in \mathcal{B}$ with $\mu(A)>0$. Then, 
$$\liminf_{N\to \infty}\frac{1}{N}\sum_{n=1}^N \mu\left( A \cap \prod_{j=1}^t T_j^{-p_{1,j}(n)}A\cap \cdots \cap \prod_{j=1}^t T_j^{-p_{k,j}(n)}A  \right)>0. $$
\end{theorem}

If $(X,\mu,T)$ is a measure preserving system, and $t\in \N$, letting $T_j=Id \times \cdots \times Id \times T \times Id \times \cdots \times Id$, be maps on $X^t$, where $T$ appears in the $j$-th coordinate, and letting $A=A_1\times \cdots \times A_t$, with $\mu(A_i)>0$, for each $i=1,\ldots,t$, we recover the following immediate corollary. 

\begin{corollary}\label{Bergelson Leibman tensor product}
Let $(X,\mu,T)$ be a measure preserving system, let $p_{i,j}$, $1\leq i\leq k$, $1\leq j\leq t$, be polynomials with rational coefficients taking on integer values and vanishing at $0$, and let $A_1,\ldots,A_t\subset X$ with $\mu(A_i)>0$, for each $i=1,\ldots,t$. Then, if $\mu^{(t)}=\mu \times \cdots \times \mu$, with the product taken $t$-times, we have
$$\liminf_{N\to \infty}\frac{1}{N}\sum_{n=1}^N \mu^{(t)}\left( A \cap \prod_{j=1}^t T^{-p_{1,j}(n)}A_j\cap \cdots \cap \prod_{j=1}^t T^{-p_{k,j}(n)}A_j  \right)>0, $$
where $A=A_1\times \cdots \times A_t$.
\end{corollary}

\begin{remark*}
We stress the difference in the use of $\prod_{j=1}^t$ between Theorem \ref{Bergelson-Leibman} and Corollary \ref{Bergelson Leibman tensor product}. In the former, it stands for the composition of functions, and in the latter, for the tensor product of sets. This distinction should be clear from context in the rest of this section.
\end{remark*}

Before proving Theorem \ref{k factor is characteristic} we 
recall the following result established in the proof of \cite[Theorem 3.6]{Ye_Yu} (see 
also \cite[Theorem 5.9]{Glasner_Huang_Shao_Ye_regionally_arithmetic_prog_nil:2020}). We give the proof with slightly more details than \cite{Ye_Yu} for the sake of completeness. 

Given a minimal infinite-step pronilsystem, $(X_{\infty},T)$, for each $k\in \N$, let $\pi_k:X_{\infty}\to X_k$ denote the factor map to its maximal $k$-step pronilfactor. Let $\mu$ be the unique invariant measure on $(X_{\infty},T)$. It is well-known (see, for example, \cite[Lemma 8.3]{Dong_Donoso_Maass_Shao_Ye_infinite_step_nil:2013}) that in this setting, $X_k=Z_k$, where $(Z_k,\mathcal{Z}_k,m_k,T)$ is the maximal $k$-step pronilfactor of $(X_{\infty},\mathcal{X}_{\infty},\mu,T)$. We can therefore consider $\pi_k$ as the continuous factor map between the measure preserving systems $(X_{\infty},\mathcal{X}_{\infty},\mu,T)$ and  $(Z_k,\mathcal{Z}_k,m_k,T)$. Note, moreover, that $\pi_k$ is an open map because it is also a topological factor map between the minimial distal systems $(X_{\infty},T)$ and $(X_k,T)$. Consider the disintegration of $\mu$ over $\pi_k$ and write $\mu=\int_{Z_k} \mu_z\ dm_k(z)$ (see, for example, \cite[Theorem 5.14]{Einsiedler_Ward11}). For $d\in \N$, let also 
$$R^{(d)}_{\pi_k}=R^{(d)}_{\pi_k}(X_{\infty}):=\{(x_1,\ldots,x_d)\in X_{\infty}^d: \pi_k(x_1)=\cdots =\pi_k(x_d) \}$$
and consider the measure
$$\lambda^d_k=\int_{Z_k} \prod_{j=1}^d \mu_z\ dm_k(z).$$

\begin{proposition}\label{support trick}
In the above context, we have that $\supp(\lambda^d_k)=R^{(d)}_{\pi_k}$.    
\end{proposition}

\begin{proof}
To see that $\supp(\lambda^d_k) \subset R^{(d)}_{\pi_k}$, 
observe 
that $R^{(d)}_{\pi_k}$ is a closed subset of $X_{\infty}^d$ with $\lambda^d_k(R^{(d)}_{\pi_k})=1$. Indeed, this follows directly from the fact that for 
$m_k$-a.e. $z\in Z_k$, $\mu_z(\pi_k^{-1}(\{z\}))=1$ and the fact that for $(x_1,\ldots,x_k)\in R^{(d)}_{\pi_k}$ there is $z\in Z_k$ such that $(x_1,\ldots,x_k)\in \pi_k^{-1}(\{z\})$. 

For the converse, \cite[Proposition 5.5]{Glasner_Huang_Shao_Ye_regionally_arithmetic_prog_nil:2020} shows that there exists a set $Y\subset Z_k$ with $m_k(Y)=1$ and such that for each $z\in Y$, $\supp(\mu_z)=\pi_k^{-1}(\{z\})$. It follows by the definition of $\lambda^d_k$ that $\prod_{j=1}^d \mu_z(\supp(\lambda^d_k))=1$, for $m_k$-a.e. $z\in Z_k$. Thus, from the definition of the support, this implies that $\prod_{j=1}^d \supp(\mu_z) \subset \supp(\lambda^d_k)$, for $m_k$-a.e. $z\in Z$. Combining these two observations we see that 
$$\prod_{j=1}^d \pi_k^{-1}(\{z\}) \subset \supp(\lambda^d_k),$$
for $m_k$-a.e. $z\in Z$. Observe that, since $(Z_k,T)$ is minimal and uniquely ergodic (strictly ergodic), $m_k$ is fully supported, and therefore, sets of full $m_k$-measure are dense in $Z_k$. 

    Finally, we observe that distality of the factor map $\pi_k:X_{\infty}\to Z_k$ -- which follows trivially from the fact that $X_{\infty}$ is distal -- means that the map $z\mapsto \pi_k^{-1}(\{z\})$ is continuous as a map from $Z_k$ to the space of nonempty closed subsets of $X_{\infty}$ endowed with the Hausdorff topology. 
    
     As $ R^{(d)}_{\pi_k}= \bigcup_{z\in Z_k} \prod_{j=1}^d \pi_k^{-1}(\{z\})$ and $\supp(\lambda^d_k)$ is a closed set that contains a dense subset of $R^{(d)}_{\pi_k}$, it follows that $R^{(d)}_{\pi_k} \subset \supp(\lambda^d_k)$, concluding the proof.  
\end{proof}

We can now prove the key intermediate result, namely Theorem 
\ref{key result}. As explained before Proposition 
\ref{support trick}, the measurable and topological $k$-step 
pronilfactors, $Z_k$ and $X_k$, of $X_{\infty}$ are the same, 
and so we use them interchangeably in the proof that follows.

\begin{proof}[Proof of Theorem \ref{k factor is characteristic}]
In the notation and assumptions of Theorem \ref{k factor is characteristic} it suffices to show that 
\begin{equation}\label{positive averages}
\limsup_{N\to \infty} \frac{1}{N}\sum_{n=1}^N \prod_{j=1}^m \mu(U_{j,0}\cap T^{-p_{j,1}(n)}U_{j,1}\cap \cdots \cap T^{-p_{j,d}(n)}U_{j,d})>0.  
\end{equation}
We shall use Theorem \ref{k+1 factor is characteristic for ergodic averages} in order to show that the averages 
$$\limsup_{N\to \infty} \frac{1}{N}\sum_{n=1}^N \prod_{j=1}^m \mu(U_{j,0}\cap T^{-p_{j,1}(n)}U_{j,1}\cap \cdots \cap T^{-p_{j,d}(n)}U_{j,d}) $$
are equal to 
$$\limsup_{N\to \infty} \frac{1}{N}\sum_{n=1}^N \prod_{j=1}^m \int_{Z_{k}} \mathbb{E}(\mathrm{1}_{U_{j,0}} | \mathcal{Z}_{k})(z) \cdot \mathbb{E}(\mathrm{1}_{U_1} | \mathcal{Z}_{k})(T^{p_{j,1}(n)}z)\cdots \mathbb{E}(\mathrm{1}_{U_{j,d}} | \mathcal{Z}_{k})(T^{p_{j,d}(n)}z)\ dm_{k}(z),$$
where $k=s-1$, with $s$ being the factor guaranteed by the aforementioned theorem. To this end, putting $$a_j(n)=\mu(U_{j,0}\cap T^{-p_{j,1}(n)}U_{j,1}\cap \cdots \cap T^{-p_{j,d}(n)}U_{j,d}) $$ 
and 
$$b_j(n)=\int_{Z_{k}} \mathbb{E}(\mathrm{1}_{U_{j,0}} | \mathcal{Z}_{k})(z) \cdot \mathbb{E}(\mathrm{1}_{U_{j,1}} | \mathcal{Z}_{k})(T^{p_{j,1}(n)}z)\cdots \mathbb{E}(\mathrm{1}_{U_{j,d}} | \mathcal{Z}_{k})(T^{p_{j,d}(n)}z)\ dm_{k}(z),$$
for $j=1,\ldots,m$, which are $1$-bounded sequences, 
our claim follows from the telescoping trick used in \eqref{eq:18} and the fact that
$$\lim_{N\to \infty} \frac{1}{N}\sum_{n=1}^N \left| a_j(n)-b_j(n) \right|=0,$$
for each $j=1,\ldots,m$. This in turn is a direct consequence of \eqref{uniformity of orthocomplement}, Theorem 
\ref{k+1 factor is characteristic for ergodic averages}, the triangle inequality and the trick used in \eqref{eq:18} once more. We thus see that in order to establish \eqref{positive averages} it suffices to show that 
\begin{equation}\label{positive averages in k+1 nilfactor}
\limsup_{N\to \infty} \frac{1}{N}\sum_{n=1}^N \prod_{j=1}^m \int_{Z_{k}} \mathbb{E}(\mathrm{1}_{U_{j,0}} | \mathcal{Z}_{k})(z) \cdot \mathbb{E}(\mathrm{1}_{U_{j,1}} | \mathcal{Z}_{k})(T^{p_{j,1}(n)}z)\cdots \mathbb{E}(\mathrm{1}_{U_{j,d}} | \mathcal{Z}_{k})(T^{p_{j,d}(n)}z)\ dm_{k}(z)>0.    
\end{equation}
To this end, since 
$$\pi_k(U_{j,0}) \cap \cdots \cap \pi_k(U_{j,d}) \neq \emptyset, $$
    we consider, for each $j=1,\ldots,m$, some point $x_j\in X_{k}=Z_{k}$ for which there exist $(y_{j,0},\ldots,y_{j,d})\in U_{j,0}\times \cdots \times U_{j,d}$ such that $\pi_{k}(y_{j,0})=\cdots=\pi_{k}(y_{j,d})=x_j$. In particular, 
$$(y_{j,0},\ldots,y_{j,d}) \in R^{(d+1)}_{\pi_{k}} \cap (U_{j,0}\times \cdots \times U_{j,d}).$$
Using Lemma \ref{support trick} this means that $(y_{j,0},\ldots,y_{j,d}) \in \supp(\lambda^{d+1}_{k})$, which, by definition, implies that 
$$\lambda^{d+1}_{k}(U_{j,0}\times \cdots \times U_{j,d})= \int_{Z_{k}} \mathbb{E}(\mathrm{1}_{U_{j,0}} | \mathcal{Z}_{k})(z) \cdot \mathbb{E}(\mathrm{1}_{U_{j,1}} | \mathcal{Z}_{k})(z)\cdots \mathbb{E}(\mathrm{1}_{U_{j,d}} | \mathcal{Z}_{k})(z)\ dm_{k}(z)>0.$$
We can therefore find $a_j>0$ and define 
$$A_j=A_{a_j}:=\{z\in Z_{k}: \mathbb{E}(\mathrm{1}_{U_{j,0}} | \mathcal{Z}_{k})(z) > a_j, \ldots, \mathbb{E}(\mathrm{1}_{U_{j,d}} | \mathcal{Z}_{k})(z)>a_j \}$$
so that $m_{k}(A_j)>0$. This is possible because 
$\mathbb{E}(\mathrm{1}_{U_{j,i}} | \mathcal{Z}_{k})\leq 1$ 
a.e. for each $0\leq i \leq d$ and $1\leq j\leq m$, and by choosing sufficiently small $a_j>0$, in view of the inequalities
\begin{align*}
    0 &<  \int_{Z_{k}} \mathbb{E}(\mathrm{1}_{U_{j,0}} | \mathcal{Z}_{k})(z) \cdot \mathbb{E}(\mathrm{1}_{U_{j,1}} | \mathcal{Z}_{k})(z)\cdots \mathbb{E}(\mathrm{1}_{U_{j,d}} | \mathcal{Z}_{k})(z)\ dm_{k}(z) \\ 
& = \int_{A_{a_j}} \mathbb{E}(\mathrm{1}_{U_{j,0}} | \mathcal{Z}_{k})(z) \cdot \mathbb{E}(\mathrm{1}_{U_{j,1}} | \mathcal{Z}_{k})(z)\cdots \mathbb{E}(\mathrm{1}_{U_{j,d}} | \mathcal{Z}_{k})(z)\ dm_{k}(z) \\
& + \int_{Z_{k}\setminus A_{a_j}} \mathbb{E}(\mathrm{1}_{U_{j,0}} | \mathcal{Z}_{k})(z) \cdot \mathbb{E}(\mathrm{1}_{U_{j,1}} | \mathcal{Z}_{k})(z)\cdots \mathbb{E}(\mathrm{1}_{U_{j,d}} | \mathcal{Z}_{k})(z)\ dm_{k}(z) \\ & \leq m_k(A_{a_j})+(1-m_k(A_{a_j}))a_j.
\end{align*} 
It therefore follows that
\begin{align*}
& \limsup_{N\to \infty} \frac{1}{N}\sum_{n=1}^N \prod_{j=1}^m \int_{Z_{k}} \mathbb{E}(\mathrm{1}_{U_{j,0}} | \mathcal{Z}_{k})(z) \cdot \mathbb{E}(\mathrm{1}_{U_{j,1}} | \mathcal{Z}_{k})(T^{p_{j,1}(n)}z)\cdots \mathbb{E}(\mathrm{1}_{U_{j,d}} | \mathcal{Z}_{k})(T^{p_{j,d}(n)}z)\ dm_{k}(z) \\ & \geq  
\liminf_{N\to \infty} \frac{1}{N}\sum_{n=1}^N \prod_{j=1}^m a_j^{d+1}  \int_{Z_{k}} \mathrm{1}_{A_j}(z) \cdot \mathrm{1}_{A_j}(T^{p_{j,1}(n)}z)\cdots \mathrm{1}_{A_j}(T^{p_{j,d}(n)}z)\ dm_{k}(z) \\ & =
\liminf_{N\to \infty} \frac{1}{N}\sum_{n=1}^N \left( \prod_{j=1}^m a_j^{d+1} \right) m_{k}^{(m)} \left(A \cap T^{-p_1(n)}A \cap \cdots \cap T^{-p_d(n)} A \right)>0,
\end{align*}
where $A=A_1 \times \cdots \times A_m$ and $T^{-p_j(n)}=T^{-p_{j,1}(n)} \times \cdots \times T^{-p_{j,d}}(n)$, for $j=1,\ldots,m$. The latter inequality follows from Corollary \ref{Bergelson Leibman tensor product}, since $m_k^{(m)}(A)=\prod_{j=1}^m m_k(A_j)>0$.
\end{proof}

\subsection{Proof of Theorem \ref{finite step characteristic for infinite step systems}}\label{final proof}

In this final part of the section we conclude the proof of 
Theorem \ref{finite step characteristic for infinite step systems}. 

\begin{proof}[Proof of Theorem \ref{finite step characteristic for infinite step systems}]
Let $(X_{\infty},T)$ be a minimal infinite-step pronilsystem and let $U_1,\ldots,U_d\subset X_{\infty}$ be nonempty open sets. 
Given a tuple of essentially distinct polynomials $p=(p_1,\ldots,p_d)\in \Z[x]^d$,
let $k$ be the integer guaranteed by 
Theorem \ref{k factor is characteristic} with 
$s=\max\{\deg(p_1),\ldots,\deg(p_d)\}$. To avoid confusion, 
we stress that the integer $k$ depends only on $d$ and $s$, 
and even though there is no assumption about the value of the polynomials in $p\in \Z[x]^d$ at $0$, we will apply Theorem \ref{k factor is characteristic} below to other polynomial families which vanish at $0$.

Let $A=R_p(\pi_k U_1,\ldots,\pi_k U_d)$ and $B=R_p(U_1,\ldots,U_d)$, where $\pi_k: X_{\infty}\to X_k$ denotes the factor map onto the maximal $k$-step pronilfactor. Clearly $B\subset A$ and therefore, by Lemma \ref{A-B pws}, it suffices to show that if $B\neq \emptyset$, then $\bigcap_{n\in F}B-f$ is syndetic for any finite $F\subset A$. Note that, by Theorem \ref{Ye Yu} (but this also follows by a special case of Theorem \ref{k factor is characteristic}) we see that when $A\neq \emptyset$, then $B\neq \emptyset$ as well, which makes the result in the case $B=\emptyset$ trivially true.

Now, let $F=\{a_1,\ldots,a_m\}\subset A$. By definition and 
the fact that $\pi_k$ is a factor map, this means that 
$$\pi_k (T^{-p_1(a_j)}U_1) \cap \cdots \cap \pi_k(T^{-p_d(a_j)}U_d) \neq \emptyset,$$
for each $j=1,\ldots,m.$
We now consider the polynomial tuples $p'_j=(p_{j,1},\ldots,p_{j,d}) \in \Z[x]^d$, where for each $j\in \{1,\ldots,m\}$ and $i\in \{1,\ldots,d\}$, we let $p_{j,i}(n)=p_i(n+a_j)-p_i(a_j)$. Since $p_1,\ldots,p_d$ are essentially distinct with degree at most $s$, we see that $p'_j$ is a tuple of essentially distinct polynomials vanishing at $0$ of maximum degree at most $s$. Hence, Theorem \ref{k factor is characteristic} implies that there exists $n_0\in \N$ such that 
\begin{equation}\label{same return time in proof}
 V_{j}:= \bigcap_{i=1}^d V_{j,i}:=\bigcap_{i=1}^d T^{-p_i(n_0+a_j)}U_i=\bigcap_{i=1}^d T^{-p_{j,i}(n_0)}(T^{-p_i(a_j)}U_i) \neq \emptyset,    
\end{equation}
for each $j=1,\ldots,m$. We next consider new tuples of polynomials $q_j=(q_{j,1},\ldots,q_{j,d})\in \Z[x]^d$, defined as 
$$q_{j,i}(n)=p_{j,i}(n+n_0+a_j)-p_{j,i}(n_0+a_j)=p_i(n+n_0+a_j)-p_i(n_0+a_j),\ \text{for}\ j=1\,\ldots,m\ \text{and}\ i=1,\ldots,d,$$
and observe that 
once again each tuple $q_j$ consists of essentially distinct 
polynomials vanishing at $0$. We then choose $\ell_1,\ldots,\ell_m\in \N$ so that the polynomials in the tuple
$$q=(q_{1,1}+\ell_1x,\ldots,q_{1,d}+\ell_1x, q_{2,1}+\ell_2x,\ldots,q_{2,d}+\ell_2x,\ldots,q_{m,1}+\ell_mx,\ldots,q_{m,d}+\ell_mx)\in \Z[x]^{d\cdot m} $$
are essentially distinct. 

From \eqref{same return time in proof}, by minimality, and 
since each $V_{j}$ is an open set, we can find $r_1,\ldots,r_m\in \N$ such that  
$$\bigcap_{j=1}^m T^{-r_j}V_j= \bigcap_{j=1}^m \bigcap_{i=1}^d T^{-r_j}V_{j,i}=\bigcap_{j=1}^m \bigcap_{i=1}^d T^{-p_i(n_0+a_j)}(T^{-r_j}U_i) \neq \emptyset.$$
In order to show that 
$$\bigcap_{j=1}^m B-a_j=\bigcap_{j=1}^m R_p(U_1,\ldots,U_d)-a_j$$
is syndetic we first show that it contains the set
$$R_q (T^{-r_1}V_{1,1},\ldots,T^{-r_1}V_{1,d},T^{-r_2}V_{2,1},\ldots,T^{-r_2}V_{2,d},\ldots,T^{-r_m}V_{m,1},\ldots,T^{-r_m}V_{m,d})+n_0.$$
Indeed, if $n\in R_q (T^{-r_1}V_{1,1},\ldots,T^{-r_1}V_{1,d},T^{-r_2}V_{2,1},\ldots,T^{-r_2}V_{2,d},\ldots,T^{-r_m}V_{m,1},\ldots,T^{-r_m}V_{m,d})$, then there is $x\in X_{\infty}$ such that for each $j=1,\ldots,m$, and each $i=1,\ldots,d$, we have that 
$$T^{q_{j,i}(n)+\ell_jn}x\in T^{-r_j}V_{j,i}.$$
Unraveling the definitions, this translates to 
$$T^{p_i(n+n_0+a_j)}(T^{\ell_jn+r_j}x)\in U_i,$$
for each $j=1,\ldots,m$ and $i=1,\ldots,d$. But this means that $$n+n_0+a_j\in R_{p}(U_1,\ldots,U_d)$$
for each $j=1,\ldots,m$, which shows the claim.

Finally, since 
$\bigcap_{j=1}^m \bigcap_{i=1}^d T^{-r_j}V_{j,i}$ is a nonempty intersection of open sets and $q\in \Z[x]^{d\cdot m}$ is a tuple of essentially distinct polynomials vanishing at $0$, it follows by Bergelson-McCutcheon's IP polynomial Szemer\'edi theorem \cite[Lemma 6.12]{Bergelson_McCutcheon00}, that 
$$R_q (T^{-r_1}V_{1,1},\ldots,T^{-r_1}V_{1,d},T^{-r_2}V_{2,1},\ldots,T^{-r_2}V_{2,d},\ldots,T^{-r_m}V_{m,1},\ldots,T^{-r_m}V_{m,d})$$
is a syndetic set. Thus, $\bigcap_{j=1}^m B-a_j$, which 
contains a shift of a syndetic set, is also syndetic, and the 
proof is complete. 
\end{proof}

We remark that the optimal integer $k\in \N$ 
appearing in Theorem \ref{k factor is characteristic}, and 
thus in Theorem 
\ref{finite step characteristic for infinite step systems} and Theorem \ref{Main theorem} -- at least following our presentation and 
proof -- depends implicitly on the polynomial tuple, in the 
sense that it only depends on the number of 
polynomials and the maximum degree among them. Observe that 
the optimal $k$ in Theorem \ref{k factor is characteristic} 
is equal to the optimal integer $k$ in 
Theorem \ref{Leibman}. But we stress that in the former, $k$ corresponds to the factor $Z_k$, whereas in the latter, $k$ corresponds to the uniformity seminorm $\nnorm{ \cdot }_{U^k},$ and thus, the characteristic factor there is $Z_{k-1}$.

However, in order to carry out our proofs -- as
becomes evident in the proof of Theorem \ref{finite step characteristic for infinite step systems} -- starting from a polynomial tuple $p=(p_1,\ldots,p_d)\in \Z[x]^d$, the only polynomial tuples we need to control by 
uniformity seminorms, in the sense of Theorem \ref{Leibman}, 
are of the form $p_a=(p_{1,a},\ldots,p_{d,a})\in \Z[x]^d$, 
where $p_{i,a}(n)=p_i(n+a)-p_i(a)$ are the normalised shifts 
of $p_i$. 
One could then optimize the integer in Theorem \ref{Main theorem} by optimizing the 
integer in Theorem \ref{Leibman} by considering the \textit{smallest characteristic factor} (a notion introduced in \cite{Fur_Weiss_96}) 
for ergodic averages of the form 
$$\frac{1}{|\Phi_N|}\sum_{n\in \Phi_N} T^{p_1(n)}f_1 \cdots T^{p_d(n)}f_d.$$
This integer or the factor associated to it, is 
commonly referred to as the 
\textit{Host-Kra complexity of the polynomial tuple $p=(p_1,\ldots,p_d)$}, a term originally introduced in \cite{Bergelson_Leibman_Lesigne}. We also remark that the 
Host-Kra complexity of the tuple $p=(p_1,\ldots,p_d)$ is the 
same as that of any of the tuples $p_{a}$; for shifting 
the polynomial variables $n\mapsto n+a$ does not affect the 
averages by the shift invariance of 
\Folner\ sequences, and the normalization by $p_i(a)$ can be 
absorbed by the functions $f_i$, without affecting 
uniformity. This last remark and the above discussion 
explain why, ultimately, the optimal $k$ in the statement of 
Theorem \ref{Main theorem} is such that $k-1$ is the Host-Kra 
complexity of $p=(p_1,\ldots,p_d)$. 

The problem of finding the true Host-Kra complexity of any 
polynomial tuple is a well-known open problem in the theory 
of multiple ergodic averages (see e.g. \cite{Frantzikinakis08, Frantzikinakis_Kra_averages_indep_poly:2006, Hernandez_Poly,  kuca, Leibman10b}).

\appendix

\section{Proof that Theorem \ref{finite step characteristic for infinite step systems} implies Theorem \ref{Main theorem}} \label{App B}

We 
already explained in Section \ref{reduction sec} 
how we can use Theorem \ref{finite step characteristic for infinite step systems} to 
deduce a similar but slightly weaker result than 
Theorem \ref{Main theorem}. In this appendix, we 
include the details 
of how to recover the full result. The reason for 
presenting this argument here is that it is a bit 
technical, it relies on slight modifications 
of arguments that were carried out in 
\cite{Huang_Shao_Ye} and \cite{Ye_Yu}, and because the
aforementioned weaker result already captures the essence of 
Theorem \ref{Main theorem}.   

\subsection{Notation and a reduction}

Recall that we fix $d\in \N$ and an essentially 
distinct tuple $p\in \Z[x]^d$. Our goal is to 
show that there exists $k\in \N$ such that for all minimal and invertible
topological dynamical systems $(X,T)$ and all nonempty, open 
sets $U_1,\ldots,U_d \subset X$,
\begin{equation}\label{finite set nilfactor eq app}
R_p((\pi_k U_1)^{\circ},\ldots,(\pi_k U_d)^{\circ}) \setminus R_p(U_1,\ldots,U_d)   
\end{equation}
is not piecewise syndetic, where $\pi_k:X\to X_k$ denotes the 
factor map to the maximal $k$-step pronilfactor of $(X,T)$. 

We further let $\phi_j:X_{\infty}\to X_j$ denote 
the factor map from $X_{\infty}$ 
to $X_j$, for each $j\in \N$ and as explained in Section \ref{reduction sec}, $\pi_k=\phi_k \circ \pi_{\infty}$, where $\pi_{\infty}:X\to X_{\infty}$ denotes the factor map onto the infinite-step pronilfactor. 

Just as Theorem \ref{finite step characteristic for infinite step systems} was deduced from Theorem \ref{k factor is characteristic}, in order to verify \eqref{finite set nilfactor eq app}, it suffices to show that for the same $k\in \N$ as in Theorem \ref{k factor is characteristic}, the following holds:

If $U_{j,i}\subset X$, for $i=1,\ldots,d$ and 
$j=1,\ldots,m$, are open sets such that 
$$(\pi_k(U_{j,1}))^{\circ} \cap \cdots \cap (\pi_k(U_{j,d}))^{\circ} \neq \emptyset, $$
then there is $n\in \N$ such that 
\begin{equation}\label{same return time app}
\bigcap_{i=1}^d T^{-p_{j,i}(n)}U_{j,i} \neq \emptyset,    
\end{equation}
for each $j=1,\ldots,m$, where $p_{j,i}$ are polynomials of the form $n\mapsto p_i(n+a_j)-p_i(a_j)$. 

For the rest of the argument, we fix $m\in \N$ 
and polynomials $p_{j,i}$ as above. In the present setting and notation (we highlight the use of three factor maps here, $\pi_{\infty}:X\to X_{\infty}$, $\phi_k:X_{\infty}\to X_k$ and $\pi_k=\phi_k\circ \pi_{\infty}: X\to X_k$), Theorem \ref{k factor is characteristic} says that if $V_{j,i}\subset X_{\infty}$, for $i=0,1,\ldots,d$ and 
$j=1,\ldots,m$, are open sets such that 
$$\phi_k(V_{j,0}) \cap \cdots \cap \phi_k(V_{j,d}) \neq \emptyset, $$
then there is $n\in \N$ such that 
$$V_{j,0} \cap \bigcap_{i=1}^d T^{-p_{j,i}(n)}V_{j,i} \neq \emptyset,$$
for each $j=1,\ldots,m$.

\subsection{A consequence of Theorem \ref{k factor is characteristic}}

As was remarked in the
end of the proof of \cite[Theorem 3.6]{Ye_Yu} and 
explained in the proof of \cite[Theorem 3.7]{Huang_Shao_Ye}, 
this implies the existence of a dense 
$G_{\delta}$ 
set $\Omega \subset X_{\infty}^m$ for which 
\begin{equation}\label{same return time app 2}(\phi_k^{(m\cdot d)})^{-1}(\phi_k^{(m\cdot d)})(x_1^{(d)},\ldots,x_m^{(d)}) \subset \overline{\{(T^{p_{1,1}(n)}x_1,\ldots,T^{p_{1,d}(n)}x_1,\ldots,T^{p_{m,1}(n)}x_m,\ldots,T^{p_{m,d}(n)}x_m):n\in \Z\}},
\end{equation}
for each $(x_1,\ldots,x_m)\in \Omega$. 

Indeed, it follows from the same proof as that of \cite[Lemma 3.6]{Huang_Shao_Ye} that there is a dense $G_{\delta}$ set $\Omega \subset X_{\infty}^m$ of continuity points $(x_1,\ldots,x_m)$ of the map 
$$X\mapsto 2^{X_{\infty}^{m\cdot d}},\ x\mapsto \overline{\mathcal{O}}_{m}(x_1^{\otimes d},\ldots,x_m^{\otimes d}),$$
where 
$$\overline{\mathcal{O}}_{m}(x_1^{\otimes d},\ldots,x_m^{\otimes d})=\overline{\{(T^{p_{1,1}(n)}x_1,\ldots,T^{p_{1,d}(n)}x_1,\ldots,T^{p_{m,1}(n)}x_m,\ldots,T^{p_{m,d}(n)}x_m):n\in \Z\}}$$
and $2^{X_{\infty}^{m\cdot d}}$ is endowed with the Hausdorff metric (see also \cite[Section 2.4.1]{Huang_Shao_Ye} for details). Then, fixing $(x_1,\ldots,x_m)\in \Omega$, we consider any tuple $$(y_{1,1},\ldots,y_{1,d},\ldots,y_{m,1},\ldots,y_{m,d})\in(\phi_k^{(m\cdot d)})^{-1}(\phi_k^{(m\cdot d)})(x_1^{(d)},\ldots,x_m^{(d)}),$$
and for any $\epsilon$ take open balls $B_{j,i}$ of radius $\epsilon$ around $y_{j,i}$. Then, for any balls $B_{j,0}$ of radius $\delta$ (to be determined later) around $x_j$, we have
$$\phi_k(x_j)\in \phi_k(B_{j,0})\cap \phi_k(B_{j,1})\cap \cdots \cap \phi_k(B_{j,d})$$
which implies that there exist $z_j\in B_{j,0}$ and $n\in \Z$ such that
$T^{p_{j,i}(n)}z\in B_{j,i}$, for each 
$i=1,\ldots,d$, $j=1,\ldots,m$. In particular,
$$\overline{\mathcal{O}}_{m}(z_1^{\otimes d},\ldots,z_m^{\otimes d}) \cap \prod_{j=1}^m (B_{j,1}\times \cdots \times B_{j,d}) \neq \emptyset$$
and by the continuity assumption -- taking appropriate $\delta$ with respect to $\epsilon$ --
it holds that
$$\overline{\mathcal{O}}_{m}(x_1^{\otimes d},\ldots,x_m^{\otimes d}) \cap \prod_{j=1}^m (B'_{j,1}\times \cdots \times B'_{j,d}) \neq \emptyset,$$
where $B'_{j,i}$ is a ball of radius $2\epsilon$ 
around $y_{j,i}$. This shows \eqref{same return time app 2}.

\subsection{A similar consequence of Theorem \cite[Theorem 3.4]{GKLMRR}}

A similar argument as above can be used in 
\cite[Theorem 3.4]{GKLMRR} to show that there exists a dense $G_{\delta}$ set 
$\Omega'\subset X^m$, for which 
\begin{equation}\label{same return time app 3}
(\pi_{\infty}^{(m\cdot d)})^{-1}(\pi_{\infty}^{(m\cdot d)})(x_1^{(d)},\ldots,x_m^{(d)}) \subset \overline{\{(T^{p_{1,1}(n)}x_1,\ldots,T^{p_{d,1}(n)}x_1,\ldots,T^{p_{1,m}(n)}x_m,\ldots,T^{p_{d,m}(n)}x_m):n\in \Z\}},
\end{equation}
for each $(x_1,\ldots,x_m)\in \Omega'$. Note 
that, because 
\cite[Theorem 3.4]{GKLMRR} is not stated for $\pi_{\infty}:X\to X_{\infty}$ but it works for a related open map which is a factor between almost 1-1 extensions of $X$ and $X^{\infty}$ -- originating from Veech's $O$-diagram construction \cite{veech_o} -- we also need an argument analogous to the equivalence between \cite[Theorem HSY]{Ye_Yu} and \cite[Theorem 3.1]{Ye_Yu} as in \cite[Lemma 3.2]{Ye_Yu}. 
To explain this, we recall the following commuting diagram used in the proof of \cite[Theorem A]{GKLMRR} from \cite[Theorem 3.4]{GKLMRR}: 
\[
\begin{tikzcd}[row sep=2cm, column sep=2cm]
X \arrow[d, "\pi_{\infty}"'] & X^{*} \arrow[l, "\tau"'] \arrow[d, "\pi^{*}_{\infty}"] \\
X_\infty & X_{\infty}^{*} \arrow[l, "\sigma"]
\end{tikzcd}
\]
Here $\tau$ and $\sigma$ are almost 1-1 factor maps, 
$\pi^{*}_{\infty}$ is an open factor map and $X_{\infty}$ is 
the infinite-step pronilfactor of both $X$ and $X^{*}$, with $\pi:X\to X_{\infty}$ the factor map.
Then, the proof of \cite[Theorem 3.4]{GKLMRR} implies \eqref{same return time app} for $\pi^{*}_{\infty}:X^{*}\to X^{*}_{\infty}$, and the argument of \eqref{same return time app 2} shows that there is a dense 
$G_{\delta}$ 
set $\Omega^{*} \subset {(X_{\infty}^{*})}^{m}$ for which 
\begin{equation}\label{same return time app 3*}((\pi^{*}_{\infty})^{(m\cdot d)})^{-1}((\pi^{*}_{\infty})^{(m\cdot d)})(x_1^{(d)},\ldots,x_m^{(d)}) \subset \overline{\mathcal{O}}_{m}(x_1^{\otimes d},\ldots,x_m^{\otimes d}),
\end{equation}
for each $(x_1,\ldots,x_m)\in \Omega^{*}$. Then, let 
$$\Omega'_1=\{y\in X_{\infty}: |\{\sigma^{-1}(y)\}|=1\}\ \text{and}\ \Omega'_2=\{y\in X: |\{\tau^{-1}(y)\}|=1\},$$
which are dense $G_{\delta}$ subsets by the almost 1-1 
assumption. We let $\Omega'=(\pi^{*}_{\infty})^{m}(\Omega^{*}) \cap \Omega'_2 \cap (\pi_{\infty}^m)^{-1}(\Omega'_1)$, which is a dense $G_{\delta}$ subset of $X^m$. Then \eqref{same return time app 3} follows from \eqref{same return time app 3*} similarly to the argument in \cite[Lemma 3.2]{Ye_Yu}.

\subsection{Combining these consequences and the conclusion}

Combining \eqref{same return time app 2} and 
\eqref{same return time app 3} we deduce, as in 
the proof of \cite[Theorem A]{Ye_Yu}, that there 
exists a dense $G_{\delta}$ set
$\Omega'' \subset X^m$, for which 
\begin{equation}\label{same return time app 4}(\pi_k^{(m\cdot d)})^{-1}(\pi_k^{(m\cdot d)})(x_1^{(d)},\ldots,x_m^{(d)}) \subset \overline{\{(T^{p_{1,1}(n)}x_1,\ldots,T^{p_{d,1}(n)}x_1,\ldots,T^{p_{1,m}(n)}x_m,\ldots,T^{p_{d,m}(n)}x_m):n\in \Z\}},
\end{equation}
for each $(x_1,\ldots,x_m)\in \Omega''$. 

To see this, let $\Omega''' \subset X^m_{\infty}$ be the dense $G_{\delta}$ set of continuity points (see e.g. \cite[Lemma 2.7]{Ye_Yu}) of the map $(y_1,\ldots,y_m)\mapsto (\pi_{\infty}^{(m)})^{-1}(y_1,\ldots,y_m)$ and then define $\Omega_0=\Omega' \cap (\pi_{\infty}^{(m)})^{-1}\left( \Omega \cap \Omega''' \right)$, which is a dense $G_{\delta}$ subset of $X^m$. Finally, it follows from \cite[Lemma 2.5]{Ye_Yu} that there is a dense $G_{\delta}$ set $\Omega_0'$ of $X^m_k$ such that for each $(y_1,\ldots,y_m)\in \Omega_0'$, $(\phi_k^{(m)})^{-1}(y_1,\ldots,y_m) \cap \Omega_0$ is dense $G_{\delta}$ in $(\phi_k^{(m)})^{-1}(y_1,\ldots,y_m)$. Then, $\Omega''=\Omega_0 \cap (\phi_k^{(m)})^{-1}(\Omega_0') \subset X^m$ is a dense $G_{\delta}$ set and one can verify that it satisfies \eqref{same return time app 4} similarly to the argument in the proof of \cite[Theorem A]{Ye_Yu}.

To conclude, suppose that $U_{j,i}\subset X$, for 
$i=0,1,\ldots,d$ and 
$j=1,\ldots,m$, are open sets such that 
$$(\pi_k(U_{j,0}))^{\circ} \cap \cdots \cap (\pi_k(U_{j,d}))^{\circ} \neq \emptyset.$$
Then, as $\Omega''$ is dense, and 
$$\prod_{j=1}^m \pi_k^{-1}\left( (\pi_k(U_{j,0}))^{\circ} \cap \cdots \cap (\pi_k(U_{j,d}))^{\circ} \right) $$
is a nonempty open subset of $X^m$, we can 
find $(x_1,\ldots,x_m)\in \Omega''$ such that $ \pi_k(x_j) \in \pi_k(U_{j,1}) \cap \cdots \cap \pi_k(U_{j,d})$, which means that there exist $$(y_{j,1},\ldots,y_{j,d}) \in (U_{j,1}\times \cdots \times U_{j,d})\ \text{and such that}\ \pi_k(y_{j,1})=\cdots = \pi_k(y_{j,d})=\pi_k(x_j),$$
for each $j=1,\ldots,m$. Applying \eqref{same return time app 4}, we can find $n\in \N$, such that 
$$(T^{p_{1,1}(n)}x_1,\ldots,T^{p_{1,d}(n)}x_1,\ldots,T^{p_{m,1}(n)}x_m,\ldots,T^{p_{m,d}(n)}x_m) \in U_{1,1}\times \cdots \times U_{1,d}\times \cdots \times U_{m,1}\times \cdots \times U_{m,d},$$
which gives \eqref{same return time app}.

\section{Relation between Theorem \ref{Ye Yu} and \cite[Theorem A]{Ye_Yu}} \label{appendix}

In this appendix we include the proof that Theorem 
\ref{Ye Yu} is a consequence of \cite[Theorem A]{Ye_Yu}. We 
stress that neither form of the result is necessary to our 
proofs, but analogous related formulations 
have turned out useful in other settings, e.g. the 
equivalence between Theorem \ref{Qiu} and \cite[Theorem B]{Qiu} utilised in \cite{GKLMRR}. On the other hand, the formulation of Theorem \ref{Ye Yu} naturally suggests the problem of estimating the difference between the return-time sets, which is addressed via Theorem \ref{Main theorem}.

\begin{theorem}{{\cite[Theorem A]{Ye_Yu}}}\label{YeYu original}
Let $(X,T)$ be a minimal and invertible topological dynamical system, $d\in \N$ and $p_1,\ldots,p_d$ be distinct non-constant integral polynomials vanishing at $0$. Then, there are $k\in \N$ (depending only on the
polynomials) and a dense $G_{\delta}$ set $\Omega$ of $X$ such that, for any $x\in \Omega$,
$$(\pi_k^{(d)})^{-1} \pi_k^{(d)} (x^{(d)}) \subset L_x^p:=\overline{\{ (T^{p_1(n)}x, \ldots, T^{p_d(n)}x): n\in \Z \} },$$
where $p=(p_1,\ldots,p_d)$, $X_k$ is the maximal k-step 
pronilfactor of $X$ and $\pi_k : X  \to X_k$ is the factor
map.    
\end{theorem}

We first show that Theorem \ref{YeYu original} is equivalent 
to the following result. This is elementary, but we include 
the details for completeness.

\begin{theorem}\label{YeYu2}
Let $(X,T)$ be a minimal and invertible topological dynamical system, $d\in \N$ and $p_1,\ldots,p_d$ be distinct non-constant integral polynomials vanishing at $0$. Then, there are $k\in \N$ (depending only on the
polynomials) and a dense $G_{\delta}$ set $\Omega$ of $X$ with the following property. For any $x\in \Omega$, and any nonempty open sets $V_1,\ldots,V_d \subset X$ with $\pi_k(x) \in \bigcap_{i=1}^d \pi_k(V_i)$, there is $n\in \Z$ such that
$$x \in T^{-p_1(n)}V_1 \cap \cdots \cap T^{-p_d(n)}V_d,$$
where $X_k$ is the maximal k-step 
pronilfactor of $X$ and $\pi_k : X  \to X_k$ is the factor
map.    
\end{theorem}

\begin{proof}[Proof that Theorem \ref{YeYu original} is equivalent to Theorem \ref{YeYu2}]
(Theorem \ref{YeYu original} implies Theorem \ref{YeYu2}) 
Take any $x\in \Omega$, where $\Omega$ is the dense $G_{\delta}$ set in Theorem \ref{YeYu original}. Let $V_1,\ldots,V_d \subset X$ be nonempty, open sets with $\pi_k(x) \in \bigcap_{i=1}^d \pi_k(V_i)$. Then, there are $y_1,\ldots,y_d\in X$ satisfying $\pi_k(y_1)=\cdots=\pi_k(y_d)=\pi_k(x)$ and $y_i\in V_i$, for $i=1,\ldots,d$. Hence, $(y_1,\ldots,y_d)\in (\pi_k^{(d)})^{-1} \pi_k^{(d)} (x^{(d)}) \cap (V_1\times \cdots \times V_d)$. We can thus find $n\in \Z$ such that 
$$(T^{p_1(n)}x,\ldots,T^{p_d(n)}x) \in V_1\times \cdots \times V_d,$$
which shows Theorem \ref{YeYu2}.

(Theorem \ref{YeYu2} implies Theorem \ref{YeYu original}) 
Take any $x\in \Omega$, where $\Omega$ is the dense $G_{\delta}$ set in Theorem \ref{YeYu2}. If $y_1,\ldots,y_d\in X$ are such that $\pi_k(y_1)=\cdots=\pi_k(y_d)=\pi_k(x)$, it suffices to show that for any basic open neighborhood $V_1\times \cdots \times V_d$ of $(y_1,\ldots,y_d)$ in $X^d$, there is $n\in \N$ such that 
$$(T^{p_1(n)}x,\ldots,T^{p_d(n)}x) \in V_1\times \cdots \times V_d.$$
But if $V_1,\ldots,V_d\subset X$ are open sets with $y_i\in V_i$, for $i=1,\ldots,d$, then $\pi_k(x)=\pi_k(y_i)\in \pi_k(V_i)$, which gives that $\pi_k(x) \in \bigcap_{i=1}^d \pi_k(V_i)$ and by Theorem \ref{YeYu2} the desired conclusion.
\end{proof}

We now show that Theorem \ref{YeYu2} implies Theorem \ref{Ye Yu}. This is trivial for polynomials vanishing at $0$, but note that Theorem \ref{Ye Yu} has no such assumption. An important detail which follows from the proof of Theorem \ref{YeYu original} -- and thus also holds for Theorem \ref{YeYu2} -- is that $k$ can be chosen to depend only on the number of polynomials and the maximum degree of them. In particular, we can choose the same $k$ for all polynomial tuples $(p_1,\ldots,p_d)\in \Z[x]^d$ vanishing at $0$, with maximum degree $\leq b$. Moreover, there are countably many choices of such tuples and the countable intersection of dense $G_{\delta}$ sets is also dense $G_{\delta}$, which means that we can further choose the same $\Omega$ for all these polynomial tuples.

\begin{proof}[Proof that Theorem \ref{YeYu2} implies Theorem \ref{Ye Yu}]
Let $(X,T)$ be a minimal and invertible
topological dynamical system. Let $d\in \N$, and let $p=(p_1,\ldots,p_d)\in \Z[x]^d$ be essentially distinct. Let $k\in \N$ be the integer and $\Omega \subset X$ the dense $G_{\delta}$ set guaranteed by Theorem \ref{YeYu2} (and the comments before this proof for $b=\max\{\deg(p_1),\ldots,\deg(p_d)\}$).

Now, consider nonempty, open sets $U_1,\ldots,U_d\subset X$, with 
$R_p((\pi_k U_1)^{\circ},\ldots,(\pi_k U_d)^{\circ}) \neq \emptyset$ 
and let $n\in R_p((\pi_k U_1)^{\circ},\ldots,(\pi_k U_d)^{\circ})$. 
Then, since $T$ is a homeomorphism and thus maps interiors to 
interiors, we have that
\begin{equation}\label{eq:2}
(\pi_kT^{-p_1(n)}U_1)^{\circ} \cap \cdots \cap (\pi_kT^{-p_d(n)}U_d)^{\circ} \neq \emptyset.    
\end{equation}
We consider the polynomials $q=(q_1,\ldots,q_d)\in \Z[x]^d$ defined 
via $q_i(m)=p_i(m+n)-p_i(n)$, which vanish at $0$, they are distinct and their degrees are at most $b$. In particular, the statement of Theorem \ref{YeYu2} holds for these with $\Omega$ and $k$ specified in the beginning of this proof.

Let also $V_i=T^{-p_i(n)}U_i$, for $i=1,\ldots,d$. Since $\pi_k(\Omega)$ is dense in $X_k$, \eqref{eq:2} implies that there is $x\in \Omega$, such that $\pi_k(x)\in \bigcap_{i=1}^d \pi_k(V_i)$. Then, $R_q(V_1,\ldots,V_d) \neq \emptyset$, and unraveling the definitions, we see that this gives $R_p(U_1,\ldots,U_d) \neq \emptyset$.
\end{proof}

Finally, we remark without proof -- but it is easy to deduce this using similar arguments as above -- that Theorem \ref{Ye Yu} is actually equivalent to the following result. 

\begin{theorem}\label{YeYu3}
Let $(X,T)$ be a minimal and invertible topological dynamical system, $d\in \N$ and $p_1,\ldots,p_d$ be distinct non-constant integral polynomials vanishing at $0$. Then, there is $k\in \N$ (depending only on the number and degrees of polynomials) with the following property. For any nonempty open sets $V_1,\ldots,V_d \subset X$ with $\bigcap_{i=1}^d (\pi_k(V_i) )^{\circ} \neq \emptyset$, there is $n\in \Z$ such that
$$T^{-p_1(n)}V_1 \cap \cdots \cap T^{-p_d(n)}V_d \neq \emptyset,$$
where $X_k$ is the maximal k-step 
pronilfactor of $X$ and $\pi_k : X  \to X_k$ is the factor
map.    
\end{theorem}

\bibliographystyle{abbrv}
\bibliography{Refs}

\vspace{1cm}

\end{document}